\documentclass[reqno,12pt]{amsart}
\usepackage{amssymb,amsmath,amsthm,mathtools}
\usepackage{enumerate}
\usepackage[margin=.99in]{geometry}

\def\Z{\mathbb{Z}}
\def\Q{\mathbb{Q}}
\def\R{\mathbb{R}}
\def\H{\mathbb{H}}

\def\C{\mathbb{C}}
\def\({\left(}
\def\){\right)}
\def\SL{{\rm SL}}
\def\GL{{\rm GL}}
\newcommand{\sldots}{\mathinner{{\ldotp}{\ldotp}{\ldotp}}}
\newcommand{\pfrac}[2]{\left(\frac{#1}{#2}\right)}
\newcommand{\ptfrac}[2]{\left(\tfrac{#1}{#2}\right)}

\renewcommand{\pmatrix}[4]{\left(\begin{smallmatrix}#1 & #2 \\ #3 & #4\end{smallmatrix}\right)}
\renewcommand{\bar}[1]{\overline{#1}}

\DeclareMathOperator{\sgn}{sgn}

\def\ep{\varepsilon}

\let\temp\phi
\let\phi\varphi
\let\varphi\temp

\newtheorem{theorem}{Theorem}
\newtheorem{lemma}[theorem]{Lemma}
\newtheorem{corollary}[theorem]{Corollary}
\newtheorem{proposition}[theorem]{Proposition}

\theoremstyle{remark}
\newtheorem*{remark}{Remark}

\numberwithin{equation}{section}

\begin{document}

\title[Weak harmonic Maass forms of weight $5/2$]{Weak harmonic Maass forms of weight $5/2$ and a mock modular form for the partition function}

\date{\today}

\author{Scott Ahlgren}
\address{Department of Mathematics\\
University of Illinois\\
Urbana, IL 61801} 
\email{sahlgren@illinois.edu} 

\author{Nickolas Andersen}
\address{Department of Mathematics\\
University of Illinois\\
Urbana, IL 61801} 
\email{nandrsn4@illinois.edu}
 
\thanks{The first author was supported by a grant from the Simons Foundation (\#208525 to Scott Ahlgren).}

\begin{abstract}  We construct a natural basis for the space of weak harmonic Maass forms of weight $5/2$ on the full modular group.  The nonholomorphic part of the first element of this basis encodes the values of the ordinary partition function $p(n)$.
We obtain a formula for the coefficients of the mock modular forms of weight $5/2$ in terms of regularized inner products of weakly holomorphic
modular forms of weight $-1/2$, and we obtain Hecke-type relations among these mock modular forms.
\end{abstract}

\maketitle

\section{Introduction}
A number of recent works have considered bases for spaces of weak harmonic Maass forms of small weight.
 Borcherds \cite{Borcherds:1998} and Zagier \cite{Zagier:2002} (in their study of  infinite product expansions of modular forms, among many other topics)  made use of the basis $\{f_d\}_{d>0}$ defined by $f_{-d}=q^{-d}+O(q)$ for the space of weakly holomorphic modular forms
of weight $1/2$ in the Kohnen plus space of level $4$.  
Duke, Imamo\=glu and T\'oth  \cite{DIT:CycleIntegrals}  extended this to a basis $\{f_d\}_{d\in \Z}$ for the space of weak harmonic Maass forms of the same weight and level
and interpreted the coefficients in terms of   cycle integrals of  the modular $j$-function.  In subsequent work, \cite{DIT:InnerProducts}
they constructed a similar basis in the case of weight $2$ for the full modular group, and related the coefficients of these forms to regularized
inner products of an infinite family of modular functions.  To construct these bases requires various types of Maass-Poincar\'e series. 
These have played a fundamental role in the theory of weak harmonic Maass forms (see for example \cite{Bringmann:2006}, \cite{BJO:Hilbert} among many other 
works).

Here, we will construct a natural basis for the space of weak harmonic Maass forms of weight $5/2$ on $\SL_2(\Z)$ with a certain multiplier.
To develop the necessary notation, let the 
  Dedekind eta-function be defined by 
\[\eta(z) := q^{\frac{1}{24}} \prod_{n\geq 1} (1-q^n), \quad\quad q:=e^{2\pi i z}.\]
 We have the generating function
\[
	\eta^{-1}(z) = \sum_{n\geq 0} p(n) q^{n-\frac{1}{24}},
\]
where $p(n)$ is  the ordinary partition function. 
The transformation property
\begin{equation}\label{eq:epdef}
	\eta(\gamma z) = \ep(\gamma) (cz+d)^{\frac 12} \eta(z),\quad \quad\gamma\in \SL_2(\Z)
\end{equation}
defines a multiplier system $\ep$ of weight $1/2$ on $\SL_2(\Z)$ which takes values in the $24$-th roots of unity.
Let  $M^!_{-\frac 12}(\bar{\ep})$ denote the space of weakly holomorphic modular forms  of weight $-1/2$ and multiplier $\overline\ep$ on $\SL_2(\Z)$.
Then $\eta^{-1}(z)$ is the first element of a natural basis 
$\{g_m\}$ for this space.
To construct the basis, we set $g_1 := \eta^{-1}$, and for each $m\equiv 1\bmod{24}$ we define $g_m := g_1 P(j)$, where $P(j)$ is a suitable monic polynomial in the Hauptmodul $j(z)$ such that $g_m = q^{-\frac{m}{24}} + O(q^{\frac{23}{24}})$. We list the first few examples here:
\begin{align} \label{eq:g-m-construction}
	g_1 &= \eta^{-1} \notag \\
		&= q^{-\frac{1}{24}} + q^{\frac{23}{24}} + 2 q^{\frac{47}{24}} + 3 q^{\frac{71}{24}} + 5 q^{\frac{95}{24}} + 7 q^{\frac{119}{24}} + 11 q^{\frac{143}{24}} + 15 q^{\frac{167}{24}}+ O(q^{\frac{191}{24}}), \notag \\
	g_{25}  &= \eta^{-1} (j - 745) \notag \\
		&=  q^{-\frac{25}{24}}+196885 q^{\frac{23}{24}}+21690645 q^{\frac{47}{24}}+ 886187500 q^{\frac{71}{24}} + O(q^{\frac{95}{24}}),  \\
	g_{49} &= \eta^{-1}(j^2-1489j+160511) \notag \\
		&=  q^{-\frac{49}{24}}+42790636 q^{\frac{23}{24}}+40513206272 q^{\frac{47}{24}} + 8543738297129 q^{\frac{71}{24}} + O(q^{\frac{95}{24}}). \notag
\end{align}

Denote by $H_{\frac 52}^!(\ep)$ the space of weak harmonic Maass forms of weight $5/2$ and multiplier $\ep$ on $\SL_2(\Z)$.
These are real analytic functions which transform as
\[
	F(\gamma z) = \ep(\gamma) (cz+d)^{\frac 52} F(z),\quad\quad \gamma=\pmatrix abcd\in \SL_2(\Z),
\]
 which are annihilated by the weight $5/2$ hyperbolic Laplacian
\[\Delta_\frac52 : =-y^2\(\frac{\partial^2}{\partial x^2}+\frac{\partial^2}{\partial y^2}\)+ \frac52 iy \(\frac{\partial}{\partial x}+i\frac{\partial}{\partial y}\), \qquad z=x+iy
\]
and which have at most linear exponential growth at $\infty$
(see the next section for details).
If $M^!_\frac52(\ep)$ denotes the space of weakly holomorphic modular forms of weight $5/2$ and multiplier $\ep$, then
we have $M^!_\frac52(\ep)\subseteq H_{\frac 52}^!(\ep)$. 

Here we construct a basis  $\{h_m\}$ for the space $H_{\frac 52}^!(\ep)$.
For negative $m$ we have $h_m\in M_{\frac 52}^!(\ep)$,
while for positive $m$ we have $\xi_\frac52(h_m)= \ptfrac{6}{\pi m}^\frac32  g_m$,
where
$\xi_\frac52:H_{\frac 52}^!(\ep)\rightarrow M_{-\frac 12}^!(\overline\ep)$ 
is the $\xi$-operator defined in \eqref{eq:def-xi} below.
In particular, the nonholomorphic part of $h_1$ encodes the values of the partition function.

To state the first result it will be convenient  to introduce the notation
\[
	\beta(y) = \Gamma\(-\tfrac{3}{2},\tfrac{\pi  y}{6}\),
\]
where the incomplete gamma function is given by $\Gamma(s, y):=\int_y^\infty e^{-t}t^{s-1}\, dt$.  

\begin{theorem}\label{thm:basis}
	For each $m \equiv 1\pmod {24}$, there is a unique $h_m \in H_{\frac 52}^!(\ep)$ with Fourier expansion of the form
	\[h_m(z) = q^{\frac{m}{24}} + \sum\limits_{0<n\equiv 1(24)} p_m^+(n) q^{\frac{n}{24}}\qquad  \text{ if } m<0,\]
	and
\[h_m(z)= \sum\limits_{0<n\equiv 1\,(24)} p^+_m(n) q^{\frac{n}{24}} 
				+ \, i   \beta(-my) q^{\frac{m}{24}}  \,
				+ \sum\limits_{0<n\equiv -1\,(24)} p^-_m(n) \beta(ny) q^{-\frac{n}{24}} \qquad \text{ if } m>0.
	\]
	The set $\{h_m\}$ forms a basis for $H_{\frac 52}^!(\ep)$. For $m<0$ we have $h_m \in M_{\frac 52}^!(\ep)$,
	 while for $m>0$ we have 
	\[
		\xi_\frac52(h_m)= \ptfrac{6}{\pi m}^\frac32\cdot g_m \in M_{-\frac 12}^!(\bar{\ep}).
	\]
	Furthermore, for $m\neq n$, the coefficients $p_m^+(n)$ are real.
\end{theorem}

When $m=1$, we have 
\[
	\xi_{\frac 52}(h_1) = \ptfrac6\pi^\frac32 \eta^{-1}(z)=\ptfrac6\pi^\frac32 \sum_{n\geq-1} p\ptfrac{n+1}{24}q^\frac{n}{24}.
\]
Using Lemma~\ref{xilemma} below we obtain the following corollary.

\begin{corollary} \label{thm:m=1}
The function ${\bf P}(z)$ defined by
\begin{align*}
	{\bf P}(z) := 
	\sum_{0<n\equiv 1\,(24)} p_1^+(n) q^{\frac{n}{24}}
	- \sum_{-1\leq n\equiv 23\,(24)}  p(\tfrac{n+1}{24}) n^{\frac 32} \beta(ny) q^{-\frac{n}{24}}
\end{align*}
is a weight $\tfrac 52$   weak harmonic Maass form on $\SL_2(\Z)$ with multiplier  $\ep$.
\end{corollary}
The coefficients can be computed using the 
 formula  in Proposition \ref{prop:h-basis} below.  We have
\begin{multline*}
	{\bf P}(z) = \left(-3.96\!\sldots + i \tfrac 43 \sqrt{\pi}\right) q^{\frac{1}{24}} + 111.40\!\sldots q^{\frac{25}{24}} + 254.26\!\sldots q^{\frac{49}{24}} + 86.52\!\sldots q^{\frac{73}{24}} + \cdots \\
	+ i \beta(-y) q^{\frac{1}{24}} - 1 \cdot 23^{\frac 32} \beta(23y) q^{-\frac{23}{24}} - 2 \cdot 47^{\frac 32} \beta(47y) q^{-\frac{47}{24}} - 3 \cdot 71^{\frac 32} \beta(71y) q^{-\frac{71}{24}} + \cdots.
\end{multline*}
We will construct the functions $h_m(z)$  in Section \ref{sec:construction}  using Maass-Poincar\'e series.  Unfortunately, the standard construction      does not produce any nonholomorphic     forms in $H_{\frac 52}^!(\ep)$,
and we must therefore consider derivatives of these series
with respect to an auxiliary parameter. This method was recently used 
by Duke, Imamo\=glu and T\'oth \cite{DIT:InnerProducts} in the case of weight $2$
(see also \cite{BDR:MockPeriodFunctions} and \cite{JKK:Maass}).
The construction provides an exact formula for the coefficients $p^\pm_m(n)$. In particular, when $m=1$ we obtain the famous exact formula of Rademacher \cite{Rademacher:Partition} for $p(n)$ as a corollary
(see Section~\ref{section:rademacher} for details).

Work of Bruinier and Ono \cite{BO:AlgebraicFormulas}
provides an algebraic formula for the coefficients $p^-_m(n)$ (and in particular the values of $p(n)$)
 as the trace of  certain weak Maass forms over CM points.
Forthcoming work of the second author \cite{Andersen:2014} investigates the analogous arithmetic and geometric nature of the coefficients $p_m^+(n)$.
In analogy with \cite{DIT:CycleIntegrals} and \cite{Bruinier:2011}, the coefficients are interpreted as the
real quadratic traces (i.e. sums of cycle integrals) of  weak Maass forms.

\begin{remark}
	When $m<0$, we can construct the forms $h_m$ directly as in \eqref{eq:g-m-construction}. 
	We list a few examples here. Let $j':=-q\frac d{dq}j$. Then
	\begin{align*}
		h_{-23} &= \eta \, j' \\
			&= q^{-\frac{23}{24}} - q^{\frac{1}{24}} - 196885 q^{\frac{25}{24}} - 42790636 q^{\frac{49}{24}} - 2549715506 q^{\frac{73}{24}} + O(q^{\frac{97}{24}}), \\
		h_{-47} &= h_{-23}(j - 743) \\
			&= q^{-\frac{47}{24}} - 2 q^{\frac{1}{24}} - 21690645 q^{\frac{25}{24}} - 40513206272 q^{\frac{49}{24}}  + O(q^{\frac{73}{24}}), \\
		h_{-71} &= h_{-23}(j^2-1487j+355910) \\
			&= q^{-\frac{71}{24}} - 3 q^{\frac{1}{24}} - 886187500 q^{\frac{25}{24}} - 8543738297129 q^{\frac{49}{24}} + O(q^{\frac{73}{24}}).
	\end{align*}
Together with the family \eqref{eq:g-m-construction}, these form a ``grid" in the sense of Guerzhoy \cite{Guerzhoy:2009} or Zagier \cite{Zagier:2002}
(note that the integers appearing as coefficients are the same up to sign as those in \eqref{eq:g-m-construction}).
\end{remark}

For $m>0$, the formula for $p^+_m(n)$ which results from the construction is an infinite series whose terms are Kloosterman sums multiplied by a derivative of the $J$-Bessel function in its index. 
Here we give an alternate interpretation of these coefficients involving the regularized Petersson inner product,
in analogy with  \cite{DIT:InnerProducts} and \cite{Duke:2011a}.

For modular forms $f$ and $g$ of weight $k$, define
\begin{equation} \label{eq:def-reg-inner-product}
	\langle f, g \rangle_\text{reg} = \lim_{Y\to \infty} \int_{\mathcal{F}(Y)} f(z) \bar{g(z)} y^k\,  \tfrac{dx\,dy}{y^2},
\end{equation}
provided that this limit exists. Here $\mathcal{F}(Y)$ denotes the usual fundamental domain for $\SL_2(\Z)$ truncated at height $Y$. 
\begin{theorem}\label{thm:innerprod}
	For   positive $m,n\equiv 1 \pmod{24}$ with $m\neq n$ we have
\[
	p^+_m(n) := \ptfrac{6}{\pi m}^{\frac 32} \cdot \langle g_m, g_n \rangle_{\rm reg}.
\]
\end{theorem}
\noindent We note that when $n=m$ the integral defining this inner product  
does not converge.

The following is an immediate corollary of Theorem~\ref{thm:innerprod}.
\begin{corollary} \label{cor:symmetry}
For positive $m,n\equiv 1 \pmod{24}$
we have 
\[
	n^\frac32 \cdot p_n^+(m)=m^\frac32\cdot  p_m^+(n).
\]
\end{corollary}

There are also Hecke relations among the forms $h_m$. 
Let $T_{\frac 52}(\ell^2)$ denote the Hecke operator of index $\ell^2$  on $H_{\frac 52}^!(\ep)$  
 (see Section \ref{sec:hecke} for definitions).
\begin{theorem}\label{thm:hecke}
For any  $m\equiv 1\pmod{24}$ and for any prime $\ell\geq 5$ we have

	\[
		h_m \big| T_{\frac 52}(\ell^2) = \ell^3 h_{\ell^2 m} + \ptfrac{3m}\ell  \ell  h_m.
	\]
\end{theorem}

Using Theorem~\ref{thm:hecke} it is possible to deduce many relations among the coefficients $p_m^+(n)$.  We 
record a typical example as a corollary.
\begin{corollary}\label{cor:hecke}
If $m,n\equiv 1\pmod{24}$ and $\ell\geq 5$ is a prime with $\pfrac{m}\ell=\pfrac{n}{\ell}$, then 
	\[
		p_m^+(\ell^2 n) = \ell^3  p_{\ell^2 m}^+(n).
	\]
\end{corollary}

\begin{remark}  It is possible to derive results analogous to Theorem~\ref{thm:hecke}  and Corollary~\ref{cor:hecke}
involving the   operators
$T_\frac52\(\ell^{2k}\)$ for any $k$ (see, for example,  \cite{Ahlgren:SingularModuli} or \cite{Ahlgren:2012}). For brevity, we will not record these statements here.
\end{remark}

In the next section we provide some background material.  In Section~3 we adapt the method of \cite{DIT:InnerProducts} to construct the basis described in Theorem~1.  The last three sections contain the proofs of the remaining results.

\subsection*{Acknowledgements} We thank Kathrin Bringmann, Jan Bruinier,  and Karl Mahlburg for their comments.

\section{Weak harmonic Maass forms}
We require some preliminaries on  weak harmonic Maass forms (see for example  \cite{Bruinier:2004}, \cite{Ono:HMFs}, or \cite{Zagier:2009} for further details).
For convenience, we set $\Gamma:=\SL_2(\Z)$.
If $k\in \frac12\Z$ then
we say that  $f$ has weight $k$ and multiplier $\nu$ for $\Gamma$ if 
\begin{equation}\label{eq:slash}
	f \big|_k \gamma = \nu(\gamma) f
\end{equation}
for every $\gamma \in \Gamma$. Here $\big|_k$ denotes the slash operator defined for $\gamma = \pmatrix abcd \in \GL_2(\Q)$ with $\det(\gamma)>0$ by
\[
	f \big|_k \gamma := (\det \gamma)^{\frac k2} (cz+d)^{-k} f \left( \frac{az+b}{cz+d} \right).
\]
We choose the argument of each nonzero  $\tau\in \C$  in $(-\pi, \pi]$, and we define $\tau^k$ using the principal branch of the logarithm.
For any $k\in \frac 12 \Z$, let $\xi_k$ denote the Maass-type differential operator which acts on differentiable functions $f$ on $\H$ by
\begin{equation} \label{eq:def-xi}
	\xi_k(f) = 2iy^k \bar{\frac{\partial f}{\partial \bar{z}}}.
\end{equation}
This operator satisfies
\begin{equation} \label{eq:xi-modular}
	\xi_k\left( f\big|_k \gamma \right) = (\xi_k f) \big|_{2-k} \gamma
\end{equation}
for any $\gamma \in \Gamma$. So if $f$ is modular of weight $k$, then $\xi_k(f)$ is modular of weight $2-k$. Furthermore, $\xi_k(f) = 0$ if and only if $f$ is holomorphic. We define the weight $k$ hyperbolic Laplacian $\Delta_k$ by
\[
	\Delta_k := -\xi_{2-k} \circ \xi_k=-y^2\(\frac{\partial^2}{\partial x^2}+\frac{\partial^2}{\partial y^2}\)+i k y \(\frac{\partial}{\partial x}+i\frac{\partial}{\partial y}\).
\]

In this paper, we are interested in the multiplier system  $\ep$ which is attached to  the Dedekind eta function. An explicit description of $\ep$ is given, for instance, in \cite[Section 2.8]{Iwaniec:Topics}. 
Defining $e(\alpha):=e^{2\pi i \alpha}$,
we have $\ep(\pmatrix 1101)=e(\tfrac 1{24})$ and $\ep(-I)=e(-\tfrac 14)$. For $\gamma=\pmatrix abcd$ with $c>0$, we have
\begin{equation} \label{eq:def-eta-mult}
	\ep(\gamma) = e\left( \frac{a+d-3c}{24c} - \frac{1}{2}s(d,c) \right),
\end{equation}
where $s(d,c)$ is the Dedekind sum
\begin{equation} \label{eq:def-dedekind-sum}
	s(d,c) := \sum_{r=1}^{c-1} \left( \frac rc - \frac 12 \right) \left( \frac {dr}c - \left\lfloor \frac {dr}c \right\rfloor - \frac 12 \right).
\end{equation}
If $f\neq 0$ satisfies \eqref{eq:slash} with $\nu=\ep$ then we must have  $2k\equiv 1 \pmod{4}$, since
\begin{equation} \label{eq:consistency}
	(-1)^{-k} f = f \big|_k (-I) = \ep(-I) f.
\end{equation}
In what follows, we assume this consistency condition so that the forms in question are not identically zero.

Suppose that $f:\H\to\C$ is real analytic and satisfies
\begin{equation}\label{eq:tran}
	f\big|_k \gamma = \ep(\gamma) f
\end{equation}
for all $\gamma \in \Gamma$. Then $f$ has a Fourier expansion at $\infty$ which is supported on exponents of the form $n/24$ with $n\equiv 1\pmod {24}$. If, in addition, $f$ satisfies
\begin{equation}\label{eq:harm}
	\Delta_k f = 0,
\end{equation}
then by the discussion in Section~\ref{sec:whit} below we have the Fourier expansion
\begin{equation} \label{eq:fourier-exp}
	f(z) = \sum_{n\equiv 1 \, (24)} a^+(n) q^{\frac{n}{24}} + \sum_{n\equiv -1\,(24)} a^-(n) \Gamma(1-k,\tfrac{\pi n y}{6}) q^{-\frac{n}{24}}.
\end{equation}
Let $H_k^!(\ep)$ denote the space of functions satisfying \eqref{eq:tran} and  \eqref{eq:harm} 
with the additional property that 
 only finitely many of the $a^\pm(n)$ with $n \leq 0$  in  \eqref{eq:fourier-exp} are nonzero.
We call elements of $H_k^!(\ep)$  {\it weak harmonic  Maass forms of weight $k$ and multiplier $\ep$}.  Note that
these forms are allowed to have poles in the nonholomorphic part.
Let $S_k(\ep) \subseteq M_k(\ep) \subseteq M_k^!(\ep) \subseteq H_k^!(\ep)$ denote the subspaces of cusp forms, modular forms, and weakly holomorphic modular forms, respectively.

The next lemma  follows from a computation.  Care must be taken with the two cases $m>0$ and $m<0$.
\begin{lemma}\label{xilemma}
 
For any $m$ and any constant $c$, we have
\[
	\xi_\frac52\left(c\cdot\beta(my)q^{-\frac{m}{24}}\right)=
	- \overline c\cdot \left(\tfrac{6}{\pi m}\right)^{\frac 32}q^{\frac{m}{24}}.
\]
\end{lemma}
Suppose that $f \in H_{\frac 52}^!(\ep)$ has Fourier expansion \eqref{eq:fourier-exp}. By Lemma \ref{xilemma} and \eqref{eq:xi-modular} we find that
\[
	\xi_\frac52(f) = - \sum_{n\equiv -1\,(24)} \bar{a^-(n)} (\tfrac{6}{\pi n})^{\frac 32} q^{\frac n{24}} \in M_{-\frac 12}^!(\bar{\ep}).
\]
Since $S_{-\frac 12}(\bar{\ep})=\{0\}$ and $S_{\frac 52}(\ep)=\{0\}$, we see that $f=0$ if and only if $a^\pm(n)=0$ for all $n<0$. 
Because of this, each form  $h_m$ in Theorem \ref{thm:basis} is uniquely determined by a single term.
If $m<0$ this term is   $q^{\frac {m}{24}}$ and if   $m>0$ this term is  $\beta(-my)q^{\frac{m}{24}}$.

\section{Construction of  harmonic Maass forms and proof of Theorem~\ref{thm:basis}} \label{sec:construction}
In order to construct the basis described in Theorem~\ref{thm:basis}, we first construct Poincar\'e series attached to 
the usual Whittaker functions (similar constructions can be found in \cite{DIT:CycleIntegrals, DIT:InnerProducts, BJO:Hilbert, BDR:MockPeriodFunctions, JKK:Maass, Bruinier:BorcherdsProducts} among others).
 It turns out that for positive $m$ these series are identically zero.  So we must differentiate 
with respect to an auxiliary parameter $s$ in order to obtain nontrivial forms in this case.  The construction is carried out in several 
subsections and is summarized in Proposition~\ref{prop:h-basis} below.
\subsection{Poincar\'e series}
In this section, fix $k\in \frac 12 \Z$ with $2k\equiv 1\pmod{4}$.
Suppose that $\phi(z)$ is a smooth function satisfying $\phi(z+1)=e\(\tfrac{1}{24}\)\phi(z)$. Since $\ep\(\pmatrix 1n01 \gamma\) = e\(\tfrac{n}{24}\)\ep\(\gamma\)$, the expression
\[
	\bar{\ep}\(\gamma\)(cz+d)^{-k} \phi\(\gamma z\)
\]
only depends on the coset $\Gamma_\infty \gamma$, where $\Gamma_\infty := \left\{ \pmatrix 1n01 : n\in \Z \right\}$. So the series
\[
	P(z) := \sum_{\gamma\in \Gamma_\infty \setminus \Gamma} \bar{\ep}(\gamma)(cz+d)^{-k} \phi(\gamma z)
\]
is well-defined if $\phi$ is chosen so that it converges absolutely. Each coset of $\Gamma_\infty \setminus \Gamma$ corresponds to a pair $(c,d)$ with $c\in \Z$ and $(d,c)=1$. 
Because of    \eqref{eq:consistency} the terms corresponding to $c$ and $-c$ are equal.

Let $m\equiv 1 \pmod{24}$ and define
\begin{equation}\label{eq:phimphi0}
	\phi_m(z) := \phi^0 \pfrac{y}{24}  e \pfrac{mx}{24},
\end{equation}
with $\phi^0(y)=O(y^\alpha)$ as $y\to 0$ for some $\alpha \in \R$. By comparison with the  Eisenstein series of weight 
$k+2\alpha$ on $\SL_2(\Z)$, 
we find that
\[
	P_m(z) := \frac{1}{2}\sum_{\gamma\in \Gamma_\infty \setminus \Gamma} \bar{\ep}(\gamma)(cz+d)^{-k} \phi_m(\gamma z)
\]
converges absolutely for $k>2-2\alpha$.

The function $P_m(z) - \phi_m(z)$ has polynomial growth as $y\to \infty$, and $e\(-\tfrac{x}{24}\)\(P_m(z)-\phi_m(z)\)$ is periodic with period 1. So we have the Fourier expansion
\[
	e\(-\tfrac{x}{24}\)\(P_m(z)-\phi_m(z)\) = \sum_{n\in \Z} a(n,y) e(nx)
\]
with
\[
	a(n,y) = \int_0^1 e\(-\tfrac{x}{24}\)\(P_m(z)-\phi_m(z)\)e(-nx)\,  dx.
\]
Using a standard argument (see, for example,   \cite[Proposition 3.1]{BJO:Hilbert})
we compute the Fourier coefficients $a(n,y)$ as
\begin{align*}
	a(n,y) &= \int_0^1 e\(-\tfrac{x}{24}\) \sum_{\substack{\gamma\in \Gamma_\infty \setminus \Gamma\\c>0}} \bar{\ep}(\gamma)(cz+d)^{-k} \phi(\gamma z) e(-nx)\,  dx \\
	&= \sum_{\substack{(c,d)=1\\c>0}} \frac{\bar{\ep}(\gamma)}{c^k} \int_\frac cd^{1+\frac cd} e\pfrac{-x+d/c}{24} z^{-k} \phi\left( \frac{a}{c} - \frac{1}{c^2z} \right) e\left( -nx + \frac{nd}{c} \right) \, dx \\
	&= \sum_{\substack{c>0\\d \bmod c^*}} \frac{\bar{\ep}(\gamma)}{c^k} e\pfrac{\bar{d}m+d(24n+1)}{24c} \\
	& \hphantom{=\sum_{d\bmod c}}\times \int_{-\infty}^\infty z^{-k} \phi^0\pfrac{y}{24c^2|z|^2} e\left( -\frac{mx}{24c^2|z|^2} - \(n+\tfrac{1}{24}\)x\right)\, dx.
\end{align*}
Here $  {c^*}$ indicates that the sum is restricted to residue classes coprime to $c$, and $\bar{d}$ denotes the inverse of $d$ modulo $c$.
The second equality comes from writing $\gamma z = \frac{a}{c} - \frac{1}{c^2(z+d/c)}$ and from making the change of variable $x\to x-d/c$. The last equality comes from writing  $d=d'+\ell c$ with $0\leq d'<c$ and gluing together the integrals for each $\ell$.

If we write $m=24m'+1$, then by \eqref{eq:def-eta-mult} we obtain
\begin{align*}
	\sum_{d \bmod c^*} \bar{\ep}(\gamma) e\pfrac{\bar{d}m+d(24n+1)}{24c} = \sqrt{i} K(m',n,c),
\end{align*}
where $K(m,n,c)$ denotes the Kloosterman sum
\begin{equation}\label{eq:kloos}
	K(m,n;c) := \sum_{d\bmod c^*} e^{\pi i s(d,c)} e\pfrac{\bar{d}m+dn}{c}.
\end{equation}

Therefore we have
\begin{equation}\label{eq:pmdef}
	P_m(z) = \phi_m(z) + \sum_{n\in \Z} c(n,y) e\(\(n+\tfrac{1}{24}\)x\)
\end{equation}
with
\begin{equation} \label{eq:fourier-coeff-integral}
	c(n,y) = \sqrt{i}\sum_{c>0} \frac{K(m',n;c)}{c^k} \int_{-\infty}^\infty z^{-k} \phi^0\left(\frac{y}{24c^2|z|^2}\right) e\left( -\frac{mx}{24c^2|z|^2} - \(n+\tfrac{1}{24}\)x \right)\, dx.
\end{equation}

\subsection{Whittaker functions and nonholomorphic Maass-Poincar\'e series}\label{sec:whit}
The Poincar\'e series $P_m(z)$ clearly has the desired transformation behavior; in order to construct  harmonic forms, we specialize $\phi^0(y)$ to be a function which has the desired behavior under $\Delta_k$.  
The Whittaker functions $M_{\mu,\nu}(y)$ and $W_{\mu,\nu}(y)$ are linearly independent solutions of Whittaker's differential equation
\begin{equation} \label{eq:whittaker-diff-eq}
	w''+\left( -\frac 14 + \frac{\mu}{y} + \frac{1-4\nu^2}{4y^2} \right)w=0,
\end{equation}
and are defined in terms of confluent hypergeometric functions (see \cite[\S 13.14]{nist} for definitions and properties). Using \eqref{eq:whittaker-diff-eq} we see after a computation that 
\[
	  y^{-\frac k2}  M_{\sgn(n)\frac k2,s- \frac 12}(4\pi |n| y) e(nx)\qquad \text{and} \qquad  y^{-\frac k2}  W_{\sgn(n)\frac k2,s- \frac 12}(4\pi |n| y) e(nx)
\]
are linearly independent solutions of the differential equation
\begin{equation}\label{eq:deltadiff}
	\Delta_k F(z) = \(s-\tfrac k2\)\(1-\tfrac k2-s\) F(z).
\end{equation}
At the special value $s=\frac k2$, we have (by \cite[(13.18.2), (13.18.5)]{nist})
\begin{equation} \label{eq:M-eval}
	y^{-\frac k2} M_{\sgn(n)\frac k2,\frac{k-1}2}(y) = e^{-\sgn(n)\frac y2}
\end{equation}
and
\begin{equation} \label{eq:W-eval}
	y^{-\frac k2} W_{\sgn(n)\frac k2,\frac{k-1}2}(y) = 
	\begin{cases}
		e^{-\frac y2} &\text{ if } n>0, \\
		\Gamma(1-k,y) e^{\frac y2} &\text{ if } n<0.
	\end{cases}
\end{equation}

For $m \equiv 1 \pmod{24}$ and $s\in \C$, define
\begin{equation}\label{eq:phims}
	\phi_{m,s}(z) := \left( 4\pi |m| \tfrac{y}{24} \right)^{-\frac k2} M_{\sgn(m)\frac k2,s-\frac 12}\(4\pi |m| \tfrac{y}{24}\) e\(\tfrac{mx}{24}\).
\end{equation}
Then $\phi_{m,s}(z) = O(y^{\text{Re}(s)-\frac k2})$ as $y\to 0$, so the series
\[
	P_m(z,s) := \frac{1}{2}\sum_{\gamma\in \Gamma_\infty \setminus \Gamma} \bar{\ep}(\gamma)(cz+d)^{-k} \phi_{m,s}(\gamma z)
\]
converges for $\text{Re}(s)>1$. Thus if one of the special values $s=\frac k2$ or $s=1-\frac k2$ is larger than $1$,
then $P_m(z,s)$ is harmonic at this value.
The following proposition describes the Fourier expansion of $P_{m}(z,s)$.

\begin{proposition} \label{prop:pm}
Let $k$ and $P_m(z,s)$ be as above, and suppose that $m\equiv 1\pmod{24}$ and $\textup{Re}(s)>1$.  Then we have
\begin{multline*}
	P_m(z,s) = \phi_{m,s}(z) + \sum_{n \equiv 1 \, (24)} g_{m,n}(s)L_{m,n}(s)
	 \(4\pi |m|\tfrac{y}{24}\)^{-\frac k2} W_{\sgn(n)\frac k2,s-\frac 12}\(4\pi |n| \tfrac{y}{24}\) e\(\tfrac{nx}{24}\)
\end{multline*}
where
\[
	g_{m,n}(s) := \frac{2\pi i^{-k+\frac 12} \Gamma(2s) \sqrt{|m/n|} }{\Gamma\left(s+\sgn(n)\frac k2\right)}
\]
and 
\[ 
	L_{m,n}(s) := 
		\begin{dcases}
			\sum\limits_{c>0} \frac{K(m',n';c)}{c} J_{2s-1}\left(\frac{\pi\sqrt{|mn|}}{6c}\right) &\text{ if } mn>0, \\ 
			\sum\limits_{c>0} \frac{K(m',n';c)}{c} I_{2s-1}\left(\frac{\pi\sqrt{|mn|}}{6c}\right) &\text{ if } mn<0.
		\end{dcases}
\]
Here $J_\alpha(x)$ and $I_\alpha(x)$ denote the usual Bessel functions and $K(m', n';c)$ is defined in \eqref{eq:kloos}.
\end{proposition}

\begin{proof}
 In view of  \eqref{eq:phims} and \eqref{eq:phimphi0} we take
\[
	\phi^0(y):=\left( 4\pi |m| y \right)^{-\frac k2} M_{\sgn(m)\frac k2,s-\frac 12}(4\pi |m| y).
\]
We write
\[n=24n'+1.\]
Then \eqref{eq:pmdef} becomes
\[
	P_m(z,s) = \phi_m(z,s) + \sum_{n\equiv 1 \, (24)} c(n',y,s) e\(\tfrac{n}{24}x\).
\]
Using \eqref{eq:fourier-coeff-integral} we find that
\begin{multline}\label{eq:cnprimey}
c(n', y,s)=\sqrt{i}\sum_{c>0}\frac{K(m', n'; c)}{c^k}\\
\times \int_{-\infty}^\infty z^{-k}  \pfrac{4\pi |m| y}{24c^2 |z|^2}^{-\frac k2} M_{\sgn(m)\frac k2,s-\frac 12} \pfrac{4\pi |m| y}{24c^2 |z|^2} e\left(\frac{-mx}{24c^2|z|^2} -  \frac {nx}{24} \right)\,  dx.
\end{multline}
The integral in \eqref{eq:cnprimey} can be written as 
\begin{align} 
	&\(4\pi |m| \tfrac{y}{24}\)^{-\frac k2} c^k i^{-k} \! \int_{-\infty}^\infty \ptfrac{y-ix}{y+ix}^{-\frac k2} \!\!M_{\sgn(m)\frac k2,s-\frac 12} \pfrac{4\pi |m| y}{24c^2 (x^2+y^2)} e\left(\frac{-mx}{24c^2(x^2+y^2)} - \frac {nx}{24} \right)\,  dx \notag \\
	&=: \(4\pi |m| \tfrac{y}{24}\)^{-\frac k2} c^k i^{-k} I. \label{eq:intreduce}
\end{align}
The integral $I$ is computed in \cite[p. 33]{Bruinier:BorcherdsProducts} in the case $m<0$, and  the case $m>0$ is similar. Write $x=-yu$. Then
\begin{align*}
	I= y \int_{-\infty}^\infty \pfrac{1-i u}{1+i u}^{\frac k2} M_{\frac k2,s-\frac 12} \pfrac{4\pi m}{24 c^2 y (u^2+1)} e\left( \frac{mu}{24c^2y(u^2+1)} + \frac{nyu}{24} \right)\,  du.
\end{align*}
Now let $A=\frac{ny}{24}$ and $B=\frac{m}{24c^2y}$. By \cite[p. 357]{Hejhal:SelbergTrace} we have
\begin{align*}
	I &= y \, \Gamma(2s) 
	\cdot\begin{dcases}
		\dfrac{2\pi \sqrt{|B/A|}}{\Gamma(s+\frac k2)} W_{\frac k2,s-\frac 12}(4\pi |A|) J_{2s-1}\(4\pi \sqrt{|AB|}\) &\text{ if }A>0,\\
		\dfrac{2\pi \sqrt{|B/A|}}{\Gamma(s-\frac k2)} W_{-\frac k2,s-\frac 12}(4\pi |A|) I_{2s-1}\(4\pi \sqrt{|AB|}\) &\text{ if }A<0,
	\end{dcases} \\
	&= \frac{2\pi\Gamma(2s)\sqrt{|m/n|}}{c\,\Gamma(s+\sgn(n)\frac k2)} W_{\sgn(n) \frac k2,s-\frac 12} (4\pi |n|\tfrac{y}{24})
	\cdot \begin{dcases}
		J_{2s-1}\left(\frac{\pi}{6c}\sqrt{|mn|}\right) &\text{ if } n>0,\\
		I_{2s-1}\left(\frac{\pi}{6c}\sqrt{|mn|}\right) &\text{ if } n<0.
	\end{dcases}
\end{align*}
Combining this with  the expression for $m<0$ from \cite[p. 33]{Bruinier:BorcherdsProducts}, we obtain
\[
	I = \frac{2\pi\Gamma(2s)\sqrt{|m/n|}}{c \, \Gamma(s+\sgn(n)\frac k2)} W_{\sgn(n) \frac k2,s-\frac 12} \(4\pi |n| \tfrac{y}{24}\)
	\cdot \begin{dcases}
		J_{2s-1}\left(\frac{\pi}{6c}\sqrt{|mn|}\right) &\text{ if } mn>0,\\
		I_{2s-1}\left(\frac{\pi}{6c}\sqrt{|mn|}\right) &\text{ if } mn<0.
	\end{dcases}
\]
Using this with \eqref{eq:cnprimey} and \eqref{eq:intreduce} we find that 
\begin{multline*}
	c(n', y,s)=\frac{2\pi i^{-k+\frac 12} \Gamma(2s) \sqrt{|m/n|}}{\Gamma(s+\sgn(n)\frac k2)} \(4\pi |m|\tfrac{y}{24}\)^{-\frac k2}  W_{\sgn(n)\frac k2,s-\frac 12} \(4\pi |n| \tfrac{y}{24}\) \\
	\times \sum_{c>0} \frac{K(m',n';c)}{c} \cdot
	\begin{dcases}
		J_{2s-1}\left(\frac{\pi}{6c}\sqrt{|mn|}\right) &\text{ if } mn>0,\\
		I_{2s-1}\left(\frac{\pi}{6c}\sqrt{|mn|}\right) &\text{ if } mn<0.
	\end{dcases} 
\end{multline*}
\end{proof}

\subsection{Derivatives of nonholomorphic Maass-Poincar\'e series in weight $\frac52$.}
We specialize Proposition~\ref{prop:pm} to the situation  $k=\frac 52$ and  $s= \frac k2=\frac54$.
Using \eqref{eq:M-eval} and \eqref{eq:W-eval} and noting that $g_{m,n}(s)=0$ for $n<0$,  we obtain

\begin{equation}\label{eq:pmz54}
	 P_m\(z,\tfrac 54\) = q^{\frac {m}{24}} - 2\pi\sum_{0<n\equiv 1 \, (24)} |\tfrac{n}{m}|^{\frac 34} L_{m,n}\(\tfrac 54\) q^{\frac{n}{24}}.
\end{equation}
Since $S_{\frac 52}(\ep) = \{0\}$, we see that  
\begin{equation}\label{eq:pmiszero}
P_m\(z, \tfrac54\)= 0\qquad\text{for $m>0$}.
\end{equation}

In order to construct nontrivial  forms when $m>0$, we apply the method of \cite{DIT:InnerProducts} and consider the derivative
\begin{equation}\label{eq:Qdef}
Q_m(z) := \partial_{s}|_{s=\frac 54} P_m(z,s).
\end{equation}
By \eqref{eq:deltadiff} we have
\[\Delta_\frac52\partial_s P_m(z, s)=(1-2s)P_m(z, s).\]
By \eqref{eq:pmiszero} we find that $\Delta_\frac52 Q_m(z)=0$;
it follows that 
\[Q_m(z)\in H^!_\frac52(\ep).\]

The following proposition gives the Fourier expansion of $Q_m(z)$, and is an analogue of Proposition~4 of \cite{DIT:InnerProducts}.

\begin{proposition} \label{prop:deriv}  
For $m>0$ let  $Q_m(z) $ be defined as in \eqref{eq:Qdef}.
Then we have
\begin{align*}
	Q_m(z) =& \(4\pi m \tfrac{y}{24}\)^{-\frac54} \partial_{s}|_{s=\frac 54} \left[ M_{\frac 54,s-\frac 12}\(4\pi m \tfrac{y}{24}\) - W_{\frac 54,s-\frac 12}\(4\pi m \tfrac{y}{24}\) \right] e\(\tfrac{mx}{24}\)- \tfrac{\Gamma'}{\Gamma}\ptfrac52 q^{\frac m{24}}  \\
	& - 2\pi\sum_{0<n\equiv 1\bmod 24} \ptfrac{n}{m}^{\frac34} \mathcal{L}_{m,n} \cdot q^{\frac n{24}} \\
	&- 2\pi \sum_{0>n\equiv 1\bmod 24} |\tfrac{n}{m}|^{\frac34} \Gamma\ptfrac52 L_{m,n}\ptfrac54\Gamma(-\tfrac32,4\pi |n| \tfrac{y}{24}) q^{\frac n{24}},
\end{align*}
where
\[
	\mathcal{L}_{m,n}  := \partial_s|_{s=\frac 54} \left[L_{m,n}(s)\right]=\sum_{c>0} \frac{K(m',n',c)}{c} \partial_s|_{s=\frac 54} \left[J_{2s-1}\left( \frac{\pi \sqrt{mn}}{6c} \right)\right].
\]
\end{proposition}

\begin{proof}
We compute  
\begin{multline*}
	\partial_s|_{s=\frac 54}  P_m(z,s) = \partial_s|_{s=\frac 54} \left[\phi_m(z,s)\right] \\ 
	\begin{aligned}
	&+ \sum_{n\neq 0} \partial_s|_{s=\frac 54}\left[ g_{m,n}(s)L_{m,n}(s) \right]\cdot\left(4\pi m \tfrac{y}{24}\right)^{-\frac54} W_{\sgn (n) \frac 54, \frac34} (4\pi |n| \tfrac{y}{24}) e\(\tfrac{nx}{24}\) \\
	&+ \sum_{n\neq 0} g_{m,n}\ptfrac54L_{m,n}\ptfrac54 \left(4\pi m \tfrac{y}{24}\right)^{-\frac54} \partial_s|_{s=\frac 54}\left[ W_{\sgn (n) \frac 54, s-\frac{1}{2}} \left(4\pi |n| \tfrac{y}{24}\right) \right] e\(\tfrac{nx}{24}\).
	\end{aligned}
\end{multline*}
By \eqref{eq:pmz54} and \eqref{eq:pmiszero} we see that if $m,n>0$, then
\begin{equation}\label{eq:lmniszero}
L_{m, n}\ptfrac54=
\begin{cases} 
	0 & \text{ if } m\neq n,\\
	\frac{1}{2\pi} & \text{ if } m=n.
\end{cases}\end{equation}						
Therefore the second sum reduces to 
\[
	-\(4\pi m \tfrac{y}{24}\)^{-\frac54} \partial_s|_{s=\frac 54} \left[ W_{\sgn (n) \frac 54, s-\frac{1}{2}} \left(4\pi m \tfrac{y}{24}\right) \right] e\(\tfrac{mx}{24}\).
\]

If $n>0$, then by \eqref{eq:W-eval}, the $n$-th term in the first sum is 
\[
	\partial_s|_{s=\frac 54} \left[g_{m,n}(s)L_{m,n}(s)\right] \ptfrac mn^{-\frac54} q^{\frac n{24}}.
\]
Using \eqref{eq:lmniszero} we find that 
\[
	\partial_s|_{s=\frac 54} \left[g_{m,n}(s)L_{m,n}(s)\right] = -\tfrac{\Gamma'}{\Gamma}\ptfrac52 \delta_{m,n} - 2\pi |\tfrac mn|^{\frac12} \mathcal{L}_{m,n}
\]
(where $\delta_{m,n}$ denotes the Kronecker delta function).

If $n<0$, then $g_{m,n}\pfrac 54=0$, so the       $n$-th term in the first sum is
\begin{multline*}
	-2\pi  \left|\tfrac mn\right|^\frac12 \Gamma\ptfrac52 L_{m,n}\ptfrac54 (4\pi m \tfrac{y}{24})^{-\frac54} W_{-\frac 54, \frac34} \left(4\pi |n| \tfrac{y}{24}\right) e\(\tfrac{nx}{24}\) \\
	= - 2\pi |\tfrac{n}{m}|^{\frac34} \Gamma\ptfrac52 L_{m,n}\ptfrac54 \Gamma(-\tfrac32,4\pi |n| \tfrac{y}{24}) q^{\frac n{24}}.  \qedhere
\end{multline*}
\end{proof}

We must evaluate the first two terms which appear in the expansion of $Q_m(z)$.
\begin{lemma}\label{lemma:firstterm}
Suppose that $k \notin \Z$ and that  $m>0$.  Then
\begin{multline} \label{eq:whittaker-derivative}
	 \(4\pi m \tfrac{y}{24}\)^{-\frac k2} \partial_{s}|_{s=\frac k2} \left[ M_{\frac k2,s-\frac 12}\(4\pi m \tfrac{y}{24}\) - W_{\frac k2,s-\frac 12}\(4\pi m \tfrac{y}{24}\) \right] e\(\tfrac{mx}{24}\) - \tfrac{\Gamma'}{\Gamma}(k) q^{\frac m{24}}   \\
	  = \left( \pi \cot \pi k - (-1)^k \pi \csc \pi k \right)q^{\frac m{24}} + (-1)^k \Gamma(k) \Gamma\(1-k,-4\pi m \tfrac{y}{24}\)q^{\frac m{24}}.
\end{multline}
\end{lemma}

\begin{proof}
By (13.14.33) of \cite{nist}, we have
\begin{align*}
	&\partial_s|_{s=\frac{k}{2}} \left[ M_{\frac k2,s-\frac 12}(y) - W_{\frac k2,s-\frac 12}(y) \right] \\
	& = \partial_s|_{s=\frac{k}{2}} \left[ M_{\frac k2,s-\frac 12}(y) - \frac{\Gamma(1-2s)}{\Gamma(1-s-k/2)} M_{\frac k2,s-\frac 12}(y) - \frac{\Gamma(2s-1)}{\Gamma(s-k/2)} M_{\frac k2,\frac 12-s}(y) \right] \\
	& = \partial_s|_{s=\frac{k}{2}} \left[1 - \frac{\Gamma(1-2s)}{\Gamma(1-s-k/2)}\right] M_{\frac k2,\frac {k-1}2}(y) - \partial_s|_{s=\frac{k}{2}}\left[\frac{\Gamma(2s-1)}{\Gamma(s-k/2)}\right] M_{\frac k2,\frac {1-k}{2}}(y) \\
	& = \frac{\Gamma'}{\Gamma}(1-k) M_{\frac k2,\frac {k-1}2}(y) - \Gamma(k-1) M_{\frac k2,\frac {1-k}{2}}(y).
\end{align*}
For the first term, we have (by \eqref{eq:M-eval})
\[
	y^{-\frac k2} M_{\frac k2,\frac {k-1}2}(y) = e^{-\frac y2}.
\]
For the second term, we use (13.14.2) and (13.6.5) of \cite{nist} to obtain
\[
	y^{-\frac k2} M_{\frac k2, \frac{1-k}{2}}(y) = (-1)^{k-1} (1-k) e^{-\frac y2} \left( \Gamma(1-k) - \Gamma(1-k,-y) \right).
\]
Therefore, the quantity on the left in \eqref{eq:whittaker-derivative} reduces to 
\[
  \left(\frac{\Gamma'}{\Gamma}(1-k) - (-1)^k \Gamma(k)\Gamma(1-k) - \frac{\Gamma'}{\Gamma}(k)\right) q^{\frac{m}{24}} + (-1)^k \Gamma(k) \Gamma\(1-k,-4\pi m \tfrac{y}{24}\) q^{\frac{m}{24}}.
\]
The lemma follows   after using the reflection formula
\[
	\Gamma(z)\Gamma(1-z) = \pi \csc \pi z,
\]
and its logarithmic derivative
\[
	\frac{\Gamma'}{\Gamma}(1-z) - \frac{\Gamma'}{\Gamma}(z) = \pi \cot \pi z.   \qedhere
\]
\end{proof}
\subsection{Proof of Theorem ~\ref{thm:basis}}
The next Proposition, which follows directly from Proposition~\ref{prop:deriv} and Lemma~\ref{lemma:firstterm},  summarizes the two  constructions.
\begin{proposition} \label{prop:h-basis}
	Let $P_m\(z, \tfrac54\)$ and $Q_m(z)$ be as defined in \eqref{eq:pmz54} and \eqref{eq:Qdef}.
For negative $m\equiv 1 \pmod{24}$, define
	\begin{equation}\label{hm1}
		h_m(z) := P_m\(z, \tfrac54\) = q^{\frac{m}{24}} - 2\pi \sum_{0<n\equiv 1\,(24)} |\tfrac nm|^{\frac 34} L_{m,n}\(\tfrac 54\) q^{\frac {n}{24}},
	\end{equation}
	and for positive $m\equiv 1 \pmod{24}$, define
	\begin{equation}\label{eq:hmm}\begin{aligned}
		h_m(z) := \frac1{\Gamma(\tfrac 52)} Q_m(z) = &-\tfrac 43 i \sqrt{\pi} q^{\frac m{24}}   - \tfrac 83 \sqrt{\pi} \sum_{0<n\equiv 1\,(24)} \ptfrac nm^{\frac 34} \mathcal{L}_{m,n} q^{\frac {n}{24}}  \\
		 &+ i \beta(-my)q^{\frac{m}{24}}  -2\pi \sum_{0>n\equiv 1\,(24)} |\tfrac nm|^{\frac 34} L_{m,n}\(\tfrac 54\) \beta(|n|y) q^{\frac{n}{24}}.
	\end{aligned}\end{equation}
	Then $h_m \in H_{\frac 52}^!(\ep)$, and for negative $m$ we have $h_m \in M_{\frac 52}^!(\ep)$.
\end{proposition}

\begin{proof}[Proof of Theorem~\ref{thm:basis}]
By the discussion following Lemma \ref{xilemma}, there can be at most one 
element of  $H_\frac 52^!(\ep)$ having a Fourier expansion of the form \eqref{hm1} and at most one
having an expansion of the form \eqref{eq:hmm}.  
By the same discussion we find that $\{h_m\}$ spans $f\in H_\frac 52^!(\ep)$.
For $m>0$  we have 
\[
	\xi_{\frac 52} h_m(z) =   \ptfrac{6}{\pi m}^{\frac 32} q^{-\frac m{24}} + 2\pi \sum_{0>n\equiv 1\,(24)} \left|\tfrac nm\right|^{\frac 34} L_{m,n}\ptfrac54 \ptfrac{6}{\pi |n|}^{\frac 32} q^{-\frac n{24}}.
\]
By comparing principal parts, we conclude that 
  	$\xi_{\frac 52} h_m  = \ptfrac{6}{\pi m}^{\frac 32}  g_m $.

We have 
\[
	\bar{K(m,n;c)} = K(m,n;c),
\]
which results from a computation using the relation $s(-d,c) =- s(d,c)$.
It follows from \eqref{eq:hmm} that the coefficients $p^+_m(n)$ are real for $m\neq n$,
and that $p^+_m(m)$ is not real for  $m>0$ (cf. the example in the first section).
\end{proof}

\section{Rademacher's formula for $p(n)$}\label{section:rademacher}
  Writing $n=24n'-1$, we obtain
\[
	\xi_k(h_1(z)) = {\ptfrac{6}{\pi}}^{\frac 32}  q^{-\frac1{24}} + \sum_{n'>0} 2\pi n^{-\frac 34} \ptfrac{6}{\pi}^{\frac 32}  \sum_{c>0} \frac{K(0,-n';c)}{c} I_{\frac 32}\(\frac{\pi\sqrt{24n'-1}}{6c}\)  q^{n'-\tfrac{1}{24}}.
\]
The $I$-Bessel function satisfies
\[
	I_{\frac 32} (x) = \sqrt{\frac{2}{\pi}} \left( \frac{x\cosh x-\sinh x}{x^{3/2}} \right) = \sqrt{\frac{2}{\pi}} x^{1/2} \frac{d}{dx}\pfrac{\sinh x}{x},
\]
so that
\[
	I_{\frac 32} \left( \tfrac{\pi}{6} \sqrt{\tfrac 23} \sqrt{n-\tfrac{1}{24}} \right) = \frac{1}{2\sqrt{2}} \frac{c^\frac 32}{\pi^2} (24n-1)^{\frac 34} \frac{d}{dn} \left( \frac{\sinh\left(\tfrac{\pi}{c} \sqrt {\tfrac23} \sqrt{n-\tfrac{1}{24}}\right)}{\sqrt{n-\tfrac{1}{24}}} \right).
\]
Together with Theorem~\ref{thm:basis} this gives 
\begin{align*}
	\eta^{-1}(z)&=\ptfrac{\pi}{6}^{\frac 32} \xi_k(h_1(z)) \\
	 &= q^{-1/24} + \frac{1}{\pi \sqrt2} \sum_{n>0} \sum_{c>0} K(0,-n;c)\sqrt{c} \frac{d}{dn} \left( \frac{\sinh\left(\tfrac{\pi}{c} \sqrt {\tfrac23} \sqrt{n-\tfrac{1}{24}}\right)}{\sqrt{n-\tfrac{1}{24}}} \right) q^{n-\frac{1}{24}}.
\end{align*}
Equating coefficients of $q^{n-\frac {1}{24}}$ gives Rademacher's exact formula
\[
	p(n) = \frac{1}{\pi \sqrt2} \sum_{c>0} K(0,-n;c)\sqrt{c} \frac{d}{dn} \left( \frac{\sinh\left(\tfrac{\pi}{c} \sqrt {\tfrac23} \sqrt{n-\tfrac{1}{24}}\right)}{\sqrt{n-\tfrac{1}{24}}} \right).
\]

\section{Proof of Theorem~\ref{thm:innerprod}}
We proceed as  in \cite{DIT:InnerProducts} and  \cite{Duke:2011a}.
Let $\mathcal F(Y)$ denote the usual fundamental domain for $\SL_2(\Z)$, truncated at height $Y$.
We have the following (see, for example, \cite[Section~9]{Borcherds:1998}).
\begin{lemma}\label{innerprodlemma}
Suppose that $k\in \frac12\Z$, that $h\in H_{2-k}^!(\ep)$, and that $g\in M_k^!(\overline\ep)$.
Then we have 
\[ 
	\int_{\mathcal F(Y)}g(z)\overline{\xi_{2-k}h(z)} \cdot y^k\, \frac{dxdy}{y^2}=\int_{-\frac12+iY}^{\frac12+iY}g(z)h(z)\, dz.
\]
\end{lemma}
\begin{proof}
We have
\[ 
	g(z)\overline{\xi_{2-k}h(z)} \cdot y^k\, \frac{dxdy}{y^2}=g(z)\frac{\partial h}{\partial\overline z}\, dz\wedge d\overline z
= -d\left(g(z)h(z)dz\right).
 \]
The result follows from Stokes' theorem.
\end{proof}

For positive $n\equiv 1\pmod{24}$ we write the forms $g_n$ from \eqref{eq:g-m-construction} as
\[g_n(z)=q^{-\frac {n}{24}}+\sum \limits_{0<j\equiv 23(24)}c_n(j)q^\frac{j}{24}.\]
To prove Theorem~\ref{thm:innerprod}, we compute 
\begin{align}\label{eq:bigint}
&\int_{-\frac12+iY}^{\frac12+iY}g_n(z)h_m(z)\, dz\notag\\
&=\int_{-\frac12+iY}^{\frac12+iY}\(\, q^{-\frac {n}{24}}+\!\!\!\!\sum \limits_{0<j\equiv 23(24)} \!\!\! c_n(j)q^\frac{j}{24}\) \\
&\times \! \left(\!-\tfrac 43 i\sqrt{\pi}  q^{\frac m{24}} - \tfrac {8\sqrt{\pi}}3 \!\!\! \!\!   \sum_{0<k\equiv 1\,(24)}\!\!  \!\!\!\ptfrac km^{\frac 34} \mathcal{L}_{m,k} q^{\frac {k}{24}} \notag 
 + i \beta(-my)q^{\frac{m}{24}} - 2\pi \!\! \!\!\!\sum_{0>k\equiv 1\,(24)} \!\! \!\!\!|\tfrac km|^{\frac 34} L_{m,k}\(\tfrac 54\) \beta(|k|y) q^{\frac{k}{24}}\!\right)
 dz.
\end{align}

By Lemma~3.4 of \cite{Bruinier:2004} we have (for some constant $C$)
\begin{equation}\label{eq:coeffest}
|c_n(j)|\ll e^{C\sqrt{j}}, \quad \left|\mathcal{L}_{m,k}\right|\ll e^{C\sqrt{k}},\quad \text{and}\quad  \left|L_{m,k}\ptfrac54\right|\ll
e^{C\sqrt{|k|}}.
\end{equation}
From \cite[(8.11.2)]{nist} we have 
\begin{equation}\label{eq:gammaas}
\beta\(|k|Y\) = \Gamma\(-\frac32, \frac{\pi|k|Y}{6}\) \sim \pfrac{\pi|k|Y}{6}^{-\frac52}e^{-\frac{\pi|k|Y}{6}}\quad\quad
 \text{as}\quad\quad |k|Y\rightarrow\infty. 
\end{equation}

From these estimates we conclude that the expression in \eqref{eq:bigint} can be integrated term by term.
Also note that for each pair $j$, $k$ with  $j\equiv 23\pmod{24}$,  $k\equiv 1\pmod{24}$, and $j\neq -k$, we have 
\[\int_{-\frac12+iY}^{\frac12+iY}q^\frac j{24}q^\frac k{24}\, dz=0.\]
Since $n\neq m$ we obtain
\begin{align*}
\int_{-\frac12+iY}^{\frac12+iY}g_n(z)h_m(z)\, dz
&=
-\tfrac {8\sqrt{\pi}}3  \ptfrac nm^{\frac 34} \mathcal{L}_{m,n} - 2\pi \sum \limits_{0>k\equiv 1\,(24)}
c_n\(|k|\)\left|\tfrac km\right|^\frac34L_{m,k}\ptfrac54\beta\(|k|Y\)\\
&=p_m^+(n)- 2\pi \sum \limits_{0>k\equiv 1\,(24)}
c_n\(|k|\)\left|\tfrac km\right|^\frac34L_{m,k}\ptfrac54\beta\(|k|Y\).
\end{align*}
By \eqref{eq:coeffest} and \eqref{eq:gammaas} we see that the sum goes to zero as $Y\rightarrow\infty$. 
So by Lemma~\ref{innerprodlemma},  \eqref{eq:def-reg-inner-product}, and Theorem~\ref{thm:basis} we obtain
\[
	\ptfrac6{\pi m}^\frac32\langle g_n, g_m\rangle_{\rm reg}= p_m^+(n). 
\]
\qed

\section{Proof of Theorem~\ref{thm:hecke}} \label{sec:hecke}
Let $\ell\geq 5$ be prime.
On the spaces $M^!_{-\frac12}(\overline\ep)$  and $H^!_\frac52(\ep)$ we have the Hecke operators
$T_{-\frac12}(\ell^2)$ and $T_{\frac52}(\ell^2)$.
For 
\[f=\displaystyle\sum_{n\equiv 23\,(24)} a(n)q^{n/24}\in M^!_{-\frac12}(\overline\ep)\]
 we have
\[
	f\big|T_{-\frac12}(\ell^2)=\sum_{n\equiv 23\,(24)} \(a(\ell^2n)+\ell^{-2}\ptfrac{-3n}{\ell}a(n)+\ell^{-3}a(n/\ell^2)\)q^{n/24}.
\]

There are also Hecke operators on the space $H^!_\frac52(\ep)$  (see, e.g. \cite{Bruinier:2010} or \cite[Section 7]{Bruinier:2010a}).
For primes $\ell\geq 5$, our Hecke operator $T_{\frac52}(\ell^2)$  is the same as
the usual Hecke operator of index $\ell^2$  on the space $H_\frac52^!\(\Gamma_0(576),\ptfrac {12}{\bullet}\)$, conjugated by $\pmatrix{24}001$.
In particular, for 
\[g=\displaystyle\sum_{n\equiv 1\,(24)} b(n)q^{n/24}\in M^!_{\frac52}(\ep)\]
 we have
\[ 
	g\big|T_{\frac52}(\ell^2)=\sum_{n\equiv 1\,(24)} \(b(\ell^2n)+\ell\ptfrac{3n}{\ell}b(n)+\ell^{3}b(n/\ell^2)\)q^{n/24},
\]
and the action of $T_{\frac52}(\ell^2)$ on the holomorphic part of a form $h\in H^!_\frac52(\ep)$ is given by the same formula.
 For $h\in H^!_\frac52(\ep)$ we have the relation 
\begin{equation}\label{eq:heckegen}
\xi_\frac52\(h\big|T_\frac52(\ell^2)\)=\ell^3\(\xi_\frac52 h\)\big|T_{-\frac12}(\ell^2)
\end{equation}
(see \cite[Proposition~7]{Bruinier:2010a}).

For $m>0$, we have $g_m=q^{-m/24}+O(q^{1/24})$.
Since the $g_m$ form a basis for $M^!_{-\frac12}(\overline\varepsilon)$, comparing principal parts gives the formula
\begin{equation}\label{eq:gmhecke}
g_m\big|T_{-\frac12}(\ell^2)=\ell^{-3}g_{\ell^2m}+\ptfrac{3m}\ell\ell^{-2}g_m.
\end{equation}
For $m<0$ we find in a similar way that 
\[
	h_m \big|T_\frac52(\ell^2)= \ell^{3} h_{\ell^2m} + \ptfrac{3m}\ell \ell \, h_m,
\]
which gives the theorem in this case.

Suppose that $m>0$.  By Theorem $1$ we   have $h_m\big|T_\frac52(\ell^2)=\sum c_nh_n$ for some coefficients $c_n$,
from which we obtain
\[
	\xi_\frac52\(h_m\big|T_\frac52(\ell^2)\) =\sum\bar{c_n}\ptfrac 6{\pi n}^{\frac32} g_n.
\]
On the other hand, the relation \eqref{eq:heckegen}
together with \eqref{eq:gmhecke}  gives
\[
	\xi_\frac52\(h_m\big|T_\frac52(\ell^2)\)=
	\ptfrac6{\pi m}^{\frac32}\(g_{\ell^2m}+\ptfrac{3m}\ell \ell \, g_m\).
\]
Since the $\{g_m\}$ are linearly independent, we see that the only nonzero coefficients in this linear combination are  $c_{\ell^2m}=\ell^3$ and $c_m=\pfrac{3m}\ell \ell$.
The theorem follows.
\qed

\bibliographystyle{plain}
\bibliography{bibliography}

\begin{thebibliography}{10}

\bibitem{Ahlgren:SingularModuli}
S.~Ahlgren.
\newblock Hecke relations for traces of singular moduli.
\newblock {\em Bull. Lond. Math. Soc.}, 44(1):99--105, 2012.

\bibitem{Ahlgren:2012}
S.~Ahlgren and B.~Kim.
\newblock Mock modular grids and {H}ecke relations for mock modular forms.
\newblock {\em Forum Math.}, 26(4):1261--1287, 2014.

\bibitem{Andersen:2014}
N.~Andersen.
\newblock Singular invariants and coefficients of weak harmonic {M}aass forms
  of weight 5/2.
\newblock {\em Preprint}.
\newblock http://arxiv.org/abs/1410.7349.

\bibitem{Borcherds:1998}
R.~Borcherds.
\newblock Automorphic forms with singularities on {G}rassmannians.
\newblock {\em Invent. Math.}, 132(3):491--562, 1998.

\bibitem{BDR:MockPeriodFunctions}
K.~Bringmann, N.~Diamantis, and M.~Raum.
\newblock Mock period functions, sesquiharmonic {M}aass forms, and non-critical
  values of {$L$}-functions.
\newblock {\em Adv. Math.}, 233:115--134, 2013.

\bibitem{Bringmann:2006}
K.~Bringmann and K.~Ono.
\newblock The {$f(q)$} mock theta function conjecture and partition ranks.
\newblock {\em Invent. Math.}, 165(2):243--266, 2006.

\bibitem{Bruinier:BorcherdsProducts}
J.~H. Bruinier.
\newblock {\em Borcherds products on {O}(2, {$l$}) and {C}hern classes of
  {H}eegner divisors}, volume 1780 of {\em Lecture Notes in Mathematics}.
\newblock Springer-Verlag, Berlin, 2002.

\bibitem{BJO:Hilbert}
J.~H. Bruinier, P.~Jenkins, and K.~Ono.
\newblock Hilbert class polynomials and traces of singular moduli.
\newblock {\em Math. Ann.}, 334(2):373--393, 2006.

\bibitem{Bruinier:2010a}
J.~H. Bruinier and K.~Ono.
\newblock Heegner divisors, {$L$}-functions and harmonic weak {M}aass forms.
\newblock {\em Ann. of Math. (2)}, 172(3):2135--2181, 2010.

\bibitem{BO:AlgebraicFormulas}
J.~H. Bruinier and K.~Ono.
\newblock Algebraic formulas for the coefficients of half-integral weight
  harmonic weak {M}aass forms.
\newblock {\em Adv. Math.}, 246:198--219, 2013.

\bibitem{Bruinier:2011}
J.~H.~H. Bruinier, J.~Funke, and {{\"O}}. Imamo{\=g}lu.
\newblock Regularlized theta liftings and periods of modular functions.
\newblock {\em J. Reine Angew. Math., to appear.}

\bibitem{Bruinier:2004}
J.~H.~Hendrik Bruinier and J.~Funke.
\newblock On two geometric theta lifts.
\newblock {\em Duke Math. J.}, 125(1):45--90, 2004.

\bibitem{Bruinier:2010}
J.~H.~Hendrik Bruinier and O.~Stein.
\newblock The {W}eil representation and {H}ecke operators for vector valued
  modular forms.
\newblock {\em Math. Z.}, 264(2):249--270, 2010.

\bibitem{nist}
{NIST Digital Library of Mathematical Functions}.
\newblock http://dlmf.nist.gov/, Release 1.0.9 of 2014-08-29.

\bibitem{DIT:InnerProducts}
W.~Duke, {{\"O}}. Imamo{\=g}lu, and {{\'A}}. T{{\'o}}th.
\newblock Regularized inner products of modular functions.
\newblock {\em Ramanujan Journal, to appear.}

\bibitem{DIT:CycleIntegrals}
W.~Duke, {{\"O}}. Imamo{\=g}lu, and {{\'A}}. T{{\'o}}th.
\newblock Cycle integrals of the {$j$}-function and mock modular forms.
\newblock {\em Ann. of Math. (2)}, 173(2):947--981, 2011.

\bibitem{Duke:2011a}
W.~Duke, {{\"O}}. Imamo{\=g}lu, and {{\'A}}. T{{\'o}}th.
\newblock Real quadratic analogs of traces of singular moduli.
\newblock {\em Int. Math. Res. Not. IMRN}, (13):3082--3094, 2011.

\bibitem{Guerzhoy:2009}
P.~Guerzhoy.
\newblock On weak harmonic {M}aass-modular grids of even integral weights.
\newblock {\em Math. Res. Lett.}, 16(1):59--65, 2009.

\bibitem{Hejhal:SelbergTrace}
D.~Hejhal.
\newblock {\em The {S}elberg trace formula for {${\rm PSL}(2,\,{\bf R})$}.
  {V}ol. 2}, volume 1001 of {\em Lecture Notes in Mathematics}.
\newblock Springer-Verlag, Berlin, 1983.

\bibitem{Iwaniec:Topics}
H.~Iwaniec.
\newblock {\em Topics in classical automorphic forms}, volume~17 of {\em
  Graduate Studies in Mathematics}.
\newblock American Mathematical Society, Providence, RI, 1997.

\bibitem{JKK:Maass}
D.~Jeon, S.-Y. Kang, and C.~H. Kim.
\newblock Weak {M}aass-{P}oincar{\'e} series and weight 3/2 mock modular forms.
\newblock {\em J. Number Theory}, 133(8):2567--2587, 2013.

\bibitem{Ono:HMFs}
K.~Ono.
\newblock Unearthing the visions of a master: harmonic {M}aass forms and number
  theory.
\newblock In {\em Current developments in mathematics, 2008}, pages 347--454.
  Int. Press, Somerville, MA, 2009.

\bibitem{Rademacher:Partition}
H.~Rademacher.
\newblock On the partition function $p(n)$.
\newblock {\em Proc. London Math. Soc.}, 43(4):241--254, 1937.

\bibitem{Zagier:2002}
D.~Zagier.
\newblock Traces of singular moduli.
\newblock In {\em Motives, polylogarithms and {H}odge theory, {P}art {I}
  ({I}rvine, {CA}, 1998)}, volume~3 of {\em Int. Press Lect. Ser.}, pages
  211--244. Int. Press, Somerville, MA, 2002.

\bibitem{Zagier:2009}
D.~Zagier.
\newblock Ramanujan's mock theta functions and their applications (d'apr\'es
  {Z}wegers and {O}no-{B}ringmann).
\newblock {\em Ast{\'e}risque}, (326):Exp. No. 986, vii--viii, 143--164 (2010),
  2009.
\newblock S{{\'e}}minaire Bourbaki. Vol. 2007/2008.

\end{thebibliography}

\end{document}